\documentclass[12pt]{amsart}
\usepackage{amssymb}
\usepackage{amsmath}
\usepackage{amsfonts}
\usepackage{graphicx}

\usepackage{wrapfig}

plus 6pt minus 12pt
\usepackage{hyperref}
    \usepackage{aeguill}
    \usepackage{type1cm}

\usepackage[shortlabels]{enumitem}

\newtheorem{thm}{Theorem}
\newtheorem{cor}[thm]{Corollary}
\newtheorem{lem}[thm]{Lemma}

\newtheorem{prop}[thm]{Proposition}

\newtheorem{defn}[thm]{Definition}

\newtheorem{claim}[thm]{Claim}

\usepackage{amsxtra}
\usepackage{eucal}
\usepackage{mathrsfs}
\usepackage{longtable}
\usepackage{caption}

\usepackage{hyperref}
    \usepackage{aeguill}
    \usepackage{type1cm}

\newcommand{\N}{\mathbb{N}}

\newcommand{\R}{\mathbb{R}}

\newcommand{\eps}{\varepsilon}

\newcommand{\eq}[1]{\begin{equation}{#1}\end{equation}}

\newcommand{\mlt}[1]{\begin{multline}{#1}\end{multline}}
\newcommand{\mltt}[1]{\begin{multline*}{#1}\end{multline*}}

\newcommand{\set}[2]{\{{#1}\colon{#2}\}}
\newcommand{\Set}[2]{\Big\{{#1}\colon{#2}\Big\}}

\newcommand{\Leqref}[1]{\stackrel{\scriptscriptstyle{\eqref{#1}}}{\leq}}

\usepackage[usenames,dvipsnames]{color}

\usepackage[dvipsnames]{xcolor}

\begin{document}

\title[Inner and outer smooth approximation of convex hypersurfaces]{Inner and outer smooth approximation of convex hypersurfaces. When is it possible?}

\author{Daniel Azagra}
\address{Departamento de An{\'a}lisis Matem{\'a}tico y Matem\'atica Aplicada,
Facultad de Ciencias Matem{\'a}ticas, Universidad Complutense, 28040, Madrid, Spain.  
}
\email{azagra@mat.ucm.es}

\author{Dmitriy Stolyarov}
\address{Department of Mathematics and Computer Science,
St. Petersburg State University, and
St. Petersburg Department of Steklov Mathematical Institute,
St. Petersburg,
Russia}
\email{d.m.stolyarov@spbu.ru}

\date{April 15, 2022}



\begin{abstract}
Let $S$ be a convex hypersurface (the boundary of a closed convex set $V$ with nonempty interior) in $\R^n$. We prove that $S$ contains no lines if and only if for every open set $U\supset S$ there exists a real-analytic convex hypersurface $S_{U} \subset U\cap \textrm{int}(V) $. We also show that $S$ contains no rays if and only if for every open set $U\supset S$ there exists a real-analytic convex hypersurface $S_{U}\subset U\setminus V$. Moreover, in both cases, $S_U$ can be taken strongly convex.  We also establish similar results for convex functions defined on open convex subsets of $\R^n$, completely characterizing the class of convex functions that can be approximated in the $C^0$-fine topology by smooth convex functions from above or from below. We also provide similar results for~$C^1$-fine approximations.
\end{abstract}

\maketitle

\section{Main results}

The main purpose of this paper is to establish the following two results.
\begin{thm}\label{main thm on outer approximation}
Let $S=\partial V$, where $V\subset \R^n$ is a closed convex set with nonempty interior. The following statements are equivalent.
\begin{enumerate}
\item $S$ contains no rays.
\item For every open set $U\supset S$ there exists a real-analytic convex hypersurface $S_{U}\subset U\setminus V$.
\item For every open set $U\supset S$ there exists a real-analytic strongly convex hypersurface $S_{U}\subset U\setminus V$.
\end{enumerate}
\end{thm}
\begin{thm}\label{main thm on inner approximation}
Let $S=\partial V$, where $V$ is a closed convex set with nonempty interior in $\R^n$. The following statements are equivalent.
\begin{enumerate}
\item $S$ contains no lines.
\item For every open set $U\supset S$ there exists a real-analytic convex hypersurface $S_{U} \subset U\cap \mathrm{int}(V)$. 
\item For every open set $U\supset S$ there exists a real-analytic strongly convex hypersurface $S_{U} \subset U\cap \mathrm{int}(V)$. 
\end{enumerate} 
\end{thm}
\begin{cor}\label{Corollary for strict convexity}
If $S$ is a strictly convex hypersurface in $\R^n$, then for every open set $U\supset S$ there exist real-analytic strongly convex hypersurfaces $S_I$ and $S_O$ such that $S_{I} \subset U\cap \mathrm{int}(V)$ and $S_{O}\subset U\setminus V$.
\end{cor}

Results of this kind are important for extending locally concave functions, which are commonly used as Bellman functions of certain extremal problems in harmonic analysis. Results similar to Corollary~\ref{Corollary for strict convexity} may be found along the lines of Section $4$ in~\cite{StolyarovZatitskiy}. Our initial motivation comes from the need of Corollary~\ref{Corollary for strict convexity} in higher dimensional generalizations of the work of that paper.

The preceding theorems are relatively easy consequences of the following results, which, as we believe, are of independent interest in themselves.

\begin{thm}\label{main thm for lower approximation of convex functions}
Let $U\subseteq\R^n$ be a non-empty open convex set, let $f\colon U\to\R$ be convex. The following statements are equivalent.
\begin{enumerate}
\item The graph of $f$ does not contain any ray.
\item For every continuous function $\varepsilon\colon U\to (0, \infty)$ there exists a real-analytic strongly convex function $g\colon U\to\R$ such that $f-\varepsilon<g<f$.
\item For every continuous function $\varepsilon\colon U\to (0, \infty)$ there exists a convex function $g\colon U\to\R$ such that $f-\varepsilon<g<f$.
\end{enumerate}
\end{thm}

\begin{thm}\label{main thm for upper approximation of convex functions}
Let $U\subseteq\R^n$ be a non-empty open convex set, let $f\colon U\to\R$ be convex. The following statements are equivalent.
\begin{enumerate}
\item The graph of $f$ does not contain any line.
\item For every continuous function $\varepsilon\colon U\to (0, \infty)$ there exists a real-analytic strongly convex function $g\colon U\to\R$ such that $f<g<f+\varepsilon$.
\item For every continuous function $\varepsilon\colon U\to (0, \infty)$ there exists a convex function $g\colon U\to\R$ such that $f<g<f+\varepsilon$.
\end{enumerate}
\end{thm}

In order to avoid any possible ambiguity in the preceding statements, let us fix some definitions. A {\em convex hypersurface} $S$ is the boundary of a convex set with nonempty interior. Such a set $S$ will be called {\em strictly convex} provided that $S$ contains no line segments. Similarly, a convex function is strictly convex if its graph does not contain any line segment. If $U$ is a nonempty convex subset of $\R^n$, we say that a $C^2$ function $f\colon U\to\R$ is {\em strongly convex} whenever $D^2f(x)$ is strictly positive definite for every $x\in U$. A (not necessarily $C^2$) function $\varphi\colon U\to\R$ will be said to be strongly convex if for every $x\in U$ there exist $r_x>0$ and a $C^2$ strongly convex function $\psi_{x}\colon B(x, r_x)\to\R$ such that $\varphi-\psi_x$ is convex on $B(x, r_x)$. A set $S$ will be called a real-analytic convex (resp. strongly convex) hypersurface of $\R^n$ provided that there exists a real-analytic convex (resp. strongly convex) function $g\colon \R^n\to\R$ such that $S=g^{-1}(r)$ for some $r>\inf_{x\in\R^n}g(x)$ (which implies that $Dg(x)\neq 0$ for all $x\in S$).

Let us now explain what we mean by a ray in the case that $U\neq\R^n$.
The phrase {\em the graph of~$f$ contains a ray} will mean, in this paper, that there exists~$x\in U$ and~$e\in \R^d$ such that the restriction of the function~$f$ to the set~$\set{x+t e}{t\in [0,\infty)}\cap U$ is affine. Line segments of the form $[x, z):=\set{(1-t)x+tz}{ t\in [0, 1)}$, where $x\in U$ and $z\in\partial U$, are rays for us. Similarly, in the above results and what follows, in the case $U\neq\R^n$, a line in $U$ will be a nonempty intersection of $U$ with a line in $\R^n$. It is worth noting that, when $U=\R^n$, saying that the graph of a convex function $f\colon \R^n\to\R$ does not contain any line is equivalent to asserting that $f$ is {\em essentially coercive} (which means coercive up to a linear perturbation). This is a consequence of \cite[Lemma 4.2]{Azagra} or \cite[Theorem 1.11]{AzagraMudarra}.

For background about this kind of problems, see \cite{Azagra} and the references therein. In \cite{Azagra} it was proved that every convex function $f\colon U\subseteq\R^n\to\R$ and every $\varepsilon\in (0, \infty)$ there exists a real-analytic convex function $g\colon U\to\R$ such that $|f-g|\leq\varepsilon$. This result is no longer valid in general when the number $\varepsilon$ is replaced with a strictly positive continuous function, although in \cite{Azagra} it was also shown that if $f$ is {\em properly convex}, then the result is still true.\footnote{A function $f\colon U\to\R$ is properly convex provided that $f=\varphi+\ell$, with $\ell$ linear and $\varphi\colon U\to [a, b)$ convex and proper (meaning that $\varphi^{-1}[a, c]$ is compact for every $c\in [a, b)$); here $b\in\R\cup\{\infty\}$.} However, unless $U=\R^n$, proper convexity is not a necessary condition for this kind of approximation. The following result enlarges the class of functions known to admit such approximations, providing a simple geometrical characterization of the class of convex functions (defined on an arbitrary convex and open subset of $\R^n$) that can be approximated in the $C^0$-fine topology by real-analytic strictly convex functions.

\begin{cor}\label{corollary for approximation of convex functions}
Let $U\subseteq\R^n$ be a non-empty open convex set, let $f\colon U\to\R$ be convex. The following statements are equivalent.
\begin{enumerate}
\item The graph of $f$ does not contain any line.
\item For every continuous function $\varepsilon\colon U\to (0, \infty)$ there exists a real-analytic strongly convex function $g\colon U\to\R$ such that $|f-g|<\varepsilon$.
\item For every continuous function $\varepsilon\colon U\to (0, \infty)$ there exists a strictly convex function $g\colon U\to\R$ such that $|f-g|<\varepsilon$.
\end{enumerate}
\end{cor}

Our methods can be tuned to obtain $C^1$-fine approximation of $C^1$ convex functions by real-analytic convex functions. The following result improves \cite[Theorem 1.10]{Azagra}.

\begin{thm}\label{C1 fine result}
Let $U\subseteq\R^n$ be a non-empty open convex set and let $f\colon U\to\R$ be convex and of class $C^1$. The following statements are equivalent.
\begin{enumerate}
\item The graph of $f$ does not contain any line.
\item For every continuous function $\varepsilon\colon U\to (0, \infty)$ there exists a real-analytic strongly convex function $g\colon U\to\R$ such that $|f-g|<\varepsilon$ and $\|Df-Dg\|<\varepsilon$.
\end{enumerate}
\end{thm}

Section~\ref{S2} contains results on approximation of convex functions by other convex functions from below and above; here we do not care about the smoothness of functions. Section~\ref{S3} derives Theorems~\ref{main thm for lower approximation of convex functions},~\ref{main thm for upper approximation of convex functions}, and Corollary \ref{corollary for approximation of convex functions} from the results of Section~\ref{S2} and techniques of~\cite{Azagra}. Section~\ref{S4} contains the proofs Theorems~\ref{main thm on outer approximation} and~\ref{main thm on inner approximation}. The last Section~\ref{S5} is devoted to the proof of Theorem~\ref{C1 fine result}.

\section{Approximation by rough functions}\label{S2}

In the proofs of Theorems \ref{main thm for lower approximation of convex functions} and \ref{main thm for upper approximation of convex functions} we will use the following two theorems (Theorems~\ref{NoRaysTheorem} and~\ref{NoLinesTheorem} below), which we believe to be novel and of independent interest.

\begin{thm}\label{NoRaysTheorem}
Let $U\subseteq\R^n$ be a non-empty open convex set, let $f\colon U\to\R$ be convex.
The graph of~$f$ does not contain any ray if and only if for every compact subset $K$ of $U$ there exists a compact subset $C$ of $U$ such that $K\subset \mathrm{int}(C)$ and, for 
$$
\varphi(x):=\sup\{f(y)+\xi(x-y) \colon y\in U\setminus C, \xi\in\partial f(y)\},\quad x\in U,
$$
we have that
$$
\inf\{f(x)-\varphi(x) \colon x\in K\}>0.
$$
\end{thm}

\noindent Here $\partial f(x)$ stands for the the subdifferential of~$f$ at the point~$x$:
\eq{
\partial f(x) = \Set{L \ \text{is a linear function}}{\forall y\in U \quad f(y) \geq f(x) + L(y-x)},\quad x\in U.
}
Recall that the set~$\partial f(x)$ is non-empty and
\eq{
f(x) = 	\sup\Set{f(y) + L(x-y)}{y\in U, \, L\in \partial f(x)},\quad x\in U.  
}

\begin{proof}[Proof of Theorem \ref{NoRaysTheorem}]
The 'only if' implication is evident, let us prove the 'if' part. If the statement is not true, then there exist $x_0\in U$ and a sequence $(y_k)\subset U$ such that $\lim_{k\to\infty}\|y_k\|=\infty$ or $\lim_{k\to\infty} d(y_k, \partial U)=0$ (if $U\neq\R^n$), and 
\begin{equation}\label{asymptotic line}
f(x_0)=\lim_{k\to\infty} \Big(f(y_k) + \eta_k(x_0-y_k)\Big)
\end{equation}
for some $\eta_k\in\partial f(y_k)$. Denoting $$v_k:=\frac{y_k-x_0}{\|y_k-x_0\|},$$ up to passing to some subsequence, we may assume that $(v_k)$ converges to some $v_0\in\mathbb{S}^{n-1}$. Here and in what follows we use the standard Euclidean norm on~$\R^n$ and denote the unit sphere by~$\mathbb{S}^{n-1}$. Since the graph of $f$ does not contain any ray and $f$ is convex, there exist two points $z_0, w_0\in U\cap \{x_0+t v_0\colon t>0\}$ such that $\|w_0-x_0\|>\|z_0-x_0\|$ and
$$
\frac{f(z_0)-f(x_0)}{\|z_0-x_0\|} <\frac{f(w_0)-f(z_0)}{\|w_0-z_0\|}\leq L(v_0)
$$
for every $L\in\partial f(w_0)$.

Let us define $$w_k:= x_0+ |w_0-x_0|v_k, \,\,\, z_k:=x_0+ |z_0-x_0|v_k.$$
The points $w_k$ and $z_k$ may fall out of $U$ for some $k$, but for all $k$ large enough we have that $w_k, z_k\in U$. Up to extracting a subsequence, we may, thus, assume that $w_k, z_k\in U$ for all $k\in\N$. Let us also set
$$
r_k:=\frac{f(w_k)-f(z_k)}{\|w_k-z_k\|}-\frac{f(z_k)-f(x_0)}{\|z_k-x_0\|},
$$
and choose~$\xi_k\in\partial f(w_k)$ for each $k\in\N$. Note that
\eq{\label{eq5}
\frac{f(w_k) - f(z_k)}{\|w_k-z_k\|} \leq \xi_k(v_k).
}

Since $\lim_{k\to\infty}w_k=w_0$, $\lim_{k\to\infty}z_k=z_0$, and $f$ is continuous, we have that
$$
\lim_{k\to\infty}r_k= r:= \frac{f(w_0)-f(z_0)}{\|w_0-z_0\|}-\frac{f(z_0)-f(x_0)}{\|z_0-x_0\|}>0.
$$
For sufficiently large $k\in\N$, we have $\|y_k-x_0\|>\|w_k-x_0\|$, and by convexity,
$$
\xi_k(v_k)\leq\eta_k(v_k),
$$
and 
\eq{\label{Eq4}
\eta_k(v_k)\geq \frac{f(y_k)-f(w_k)}{\|y_k-w_k\|}.
}
Therefore we have
\begin{eqnarray*}
& & f(y_k)+\eta_k(x_0-y_k)= f(y_k)-\|y_k-x_0\|\eta_k(v_k)=\\
& & f(y_k)-\|y_k-w_k\|\eta_k(v_k)-\|w_k-x_0\|\eta_k(v_k) \Leqref{Eq4} \\
& & f(y_k)-f(y_k)+f(w_k)-\|w_k-x_0\|\eta_k(v_k) \leq \\
& & f(w_k)-\|w_k-x_0\| \xi_k(v_k)= \\
& & f(w_k)-\|w_k-z_k\|\xi_k(v_k)-\|z_k-x_0\|\xi_k(v_k) \Leqref{eq5} \\
& & f(w_k)+f(z_k)-f(w_k)+\|z_k-x_0\|\frac{f(z_k)-f(w_k)}{\|w_k-z_k\|}=\\
& & f(z_k)+\|z_k-x_0\| \left( -r_k -\frac{f(z_k)-f(x_0)}{\|z_k-x_0\|}\right)=\\
& & f(x_0)-r_k \|z_k-x_0\|,
\end{eqnarray*}
which implies
$$
\limsup_{k\to\infty} \Big(f(y_k)+\eta_k(x_0-y_k)\Big)\leq f(x_0)-\|z_0-x_0\|r <f(x_0),
$$
in contradiction to \eqref{asymptotic line}.
\end{proof}

\begin{cor}\label{Just convex approximation from outside]}
Let $U\subseteq\R^n$ be a non-empty open convex set, let $f\colon U\to\R$ be convex.
Let~$\eps\colon U\to \R_+$ be a strictly positive continuous function. Assume that the graph of~$f$ does not contain rays. There exists a convex function~$g\colon U \to \R$ such that
\eq{
f(x) - \eps(x) < g(x) < f(x)
}
for all~$x\in U$.
\end{cor}
The strict sign in the second inequality is important. 
\begin{proof}
Let us construct~$g$ with the formula
\eq{
g(y) = \sup\Set{f(x)+ L(y-x) - \frac{1}{2}\eps(x)}{x\in U,\quad L_x \in \partial f(x)}.
}
The function~$g$ is clearly convex. Plugging~$x := y$ into this formula, we get~$g(y) \geq f(y) - \frac{1}{2}\eps(y)>f(y)-\varepsilon(y)$. The inequality~$g(y) < f(y)$ follows from Theorem~\ref{NoRaysTheorem} and the continuity of~$\eps$.
\end{proof}

\begin{thm}\label{NoLinesTheorem}
The graph of~$f$ does not contain lines if and only if for any~$x\in U$ there exists a compact set~$C_x\subset U$ and an affine function~$A_{x}$ such that~$f(x) < A_{x}(x)$, however,
$f(y) > A_{x}(y)$ provided~$y\in U\setminus C_x$.
\end{thm}

The proof of the theorem will take some time. 

Let~$L$ be a function in the subdifferential of~$f$ at~$x$. Consider the set
\eq{\label{VDef}
V = \set{y\in U}{f(y) = f(x) + L(y-x)}.
}
Then~$V$ is a relatively closed convex subset of~$U$ (of course,~$V$ is not necessarily closed as a subset of~$\R^n$). Let~$V_\infty$ be another set,
\eq{\label{VInftyDef}
V_{\infty} = \Set{y\in U}{\exists \text{ a ray } [x,z)\subset V \text{ such that } y\in [x,z)}.
}
As usual, by a ray we understand either a classical ray (then~$z$ is an infinite point) or the segment~$[x,z)$ with~$z\in \partial U$. The set~$V_\infty$ is relatively closed in~$U$. However, in general situation, it might be non-convex.

\begin{lem}\label{ConeConvexity}
Assume~$V$ does not contain lines. Then~$V_\infty$ is convex.
\end{lem}

\begin{wrapfigure}{l}{0.57\textwidth} 
\vspace{-20pt}
  \begin{center}
    \includegraphics[width=0.57\textwidth]{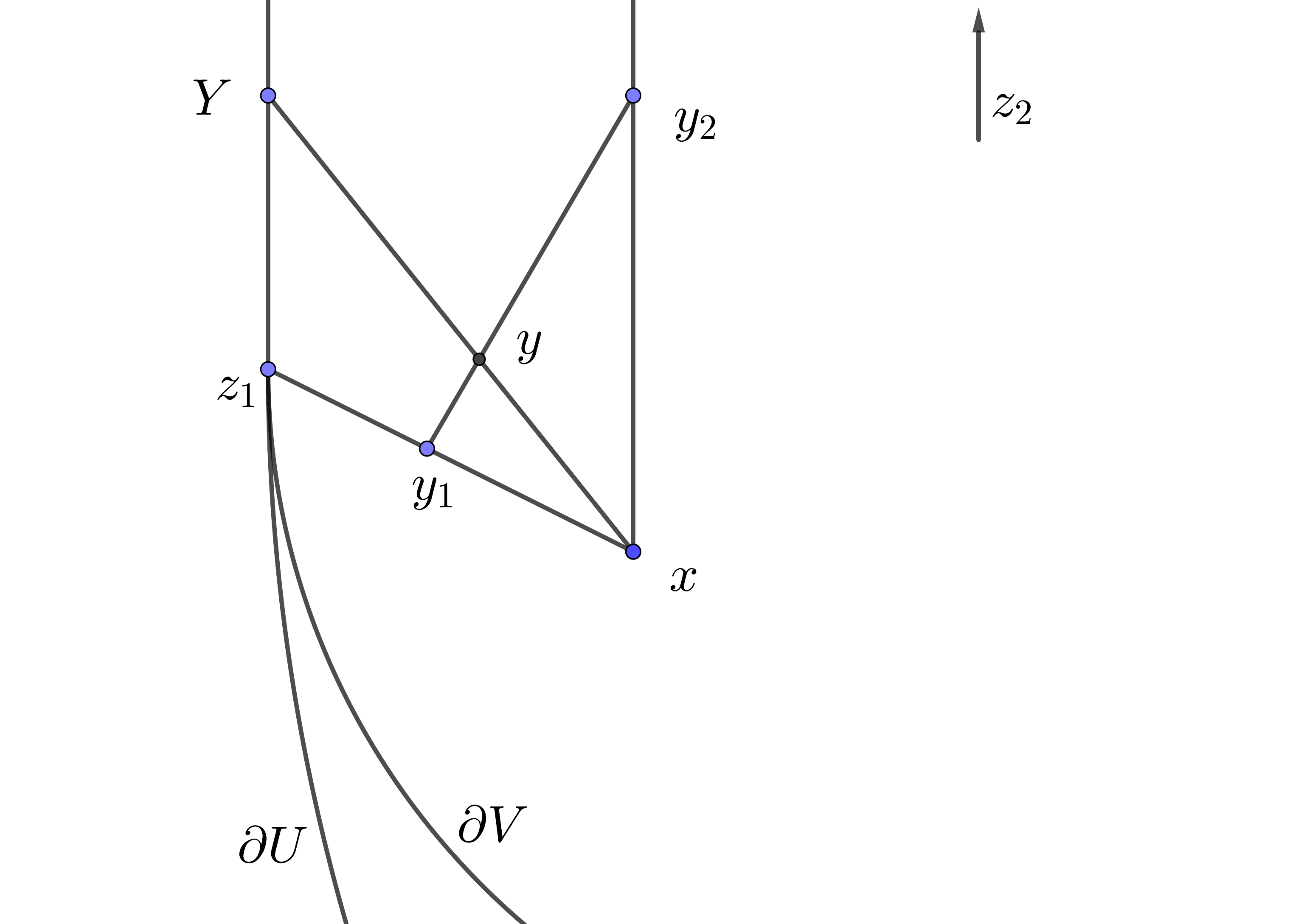}
    \caption{Illustration to the proof of Lemma~\ref{ConeConvexity}.}
\label{Ill1}
  \end{center}
  \vspace{-20pt}
  \vspace{1pt}
\end{wrapfigure} 
\emph{Proof.}
Let~$y_1$ and~$y_2$ be two points in~$V_\infty$ lying on the rays~$[x,z_1)$ and~$[x,z_2)$ correspondingly. The reasoning depends on whether~$z_1$ and~$z_2$ are finite or infinite. Let us consider the case where~$z_1$ is a finite point and~$z_2$ is infinite (this case is the most 'representative'), see Fig.~\ref{Ill1} for a visualization. Consider the ray~$(z_1,z_2)$ (which means a ray with the vertex~$z_1$ and collinear with~$[x,z_2)$). It~$(z_1,z_2)\cap U\ne \varnothing$, then this intersection is contained in~$V$ (by convexity and closedness of~$V$); this cannot happen since in such a case~$(z_1,z_2)\cap U$ is a line. 

Let~$y\in (y_1,y_2)$, we wish to show that~$y\in V_\infty$. Let~$Y = (z_1,z_2)\cap \{x+t(y-x)\colon t > 0\}$. By the above,~$Y \in \partial U$ and~$[x,Y)\subset V$. Thus,~$y\in V_\infty$.

The case where~$z_1$ and~$z_2$ are finite points is similar. The only difference is that now~$(z_1,z_2)$ is a classical segment. The case where~$z_1$ and~$z_2$ are both infinite is a little bit different (in fact, it simplifies). In this case, we do not need the assumption that~$V$ does not contain lines. We consider the ray~$\{x+t(y-x)\colon t > 0\}$ and prove directly that it belongs to~$V_\infty$ (this follows from the closedness and convexity of~$V$).
\qed

We will also need a version of the hyperplane separation theorem. We provide the proof since the construction will be used in Section~\ref{S4} below.
\begin{thm}\label{SeparationTheorem}
Let~$K$ be a convex closed cone in~$\R^n$ with the vertex at the origin. Assume~$K$ does not contain lines. There exists a linear hyperplane~$H$ such that~$H \cap K = \{0\}$.
\end{thm}
\begin{proof}
First, we note that the origin is an extreme point of~$K$. Second, consider the set~$\tilde{K} = \mathrm{conv}(K\cap \mathbb{S}^{n-1})$. This is a compact convex set that does not contain the origin (since~$K\cap \mathbb{S}^{n-1}\subset K\setminus \{0\}$,~$K\cap \mathbb{S}^{n-1}$ is compact, and~$0$ is an extreme point of~$K$). Thus, by the classical hyperplane separation theorem, there exists a hyperplane~$\tilde{H}$ that strongly separates~$\tilde{K}$ and~$0$. For example, one may consider the point~$\zeta\in \tilde{K}$ that has the smallest possible Euclidean norm and set~$\tilde{H}$ to be the midperpendicular of~$\zeta$ and the origin. Let~$H$ be the translate of~$\tilde{H}$ passing through~$0$. Then~$H\cap \tilde{K} = \varnothing$. Therefore,~$H\cap K = \{0\}$.
\end{proof}

\begin{proof}[Proof of Theorem~\ref{NoLinesTheorem}.]
If~$C_x$ and~$A_x$ as in the second condition of the theorem exist for any~$x$, then the graph of~$f$ does not contain lines. The reverse implication is less trivial.

Set~$x=0$ for convenience. Pick some~$L \in \partial f (0)$, consider the set~$V$ defined in~\eqref{VDef} and note that it does not contain lines (since the graph of~$f$ does not). Then, by Lemma~\ref{ConeConvexity},~$V_\infty$ given by~\eqref{VInftyDef} is a closed convex set that does not contain lines. Consider the minimal convex cone with the vertex~$0$ that contains~$V_\infty$ and call it~$V^*$ (note that~$V^*$ is not necessarily a subset of~$U$). This cone is also convex, closed, and does not contain lines.  By Theorem~\ref{SeparationTheorem}, there exists a hyperplane~$H$ that intersects~$V^*$ (and therefore,~$V_\infty$) at the origin only. Without loss of generality, we may assume~$H = \set{y\in \R^n}{y_n = 0}$ and that~$V_\infty$ lies in the hyperspace where~$y_n \geq 0$. We also assume~$f(0) = 0$ and~$L = 0$ (we may subtract an affine function from~$f$ and~$A_x$).

We set~$A(y) = - \eps y_n$, where~$\eps$ is a sufficiently small parameter to be specified later. 
Since~$V^*$ meets~$H$ only at the origin, there exists~$\delta > 0$ such that
\eq{\label{StrictPositivity}
V_{\infty}\setminus \{0\} \subset \Set{y\in \R^n}{y_n > \delta \|y\|}.
}
Let us call the latter set~$K_\delta$. Note that~$W = \mathrm{conv}(V\setminus K_\delta)$ is a compact set lying inside~$U$. Let~$S$ be a star-shaped (not necessarily convex) closed bounded set such that~$V\setminus K_\delta \subset \mathrm{int} S$ and~$(S\cap K_{\delta})\subset U$. We may construct the set~$S$ in the following way. Assume~$U$ contains the closed Euclidean ball of radius~$\nu$ centered at the origin. Let~$\rho$ be the distance between~$W$ and~$\partial U$. Define the function~$s\colon \mathbb{S}^{n-1} \to \R_+$ by the formula
\eq{
s(z) = \max\big(\nu, (\mu_W(z))^{-1} + \rho/2\big), \quad z\in \mathbb{S}^{n-1};
}
here~$\mu_W$ denotes the Minkowski functional of~$W$. We set~$S = \cup_{z\in \mathbb{S}^{n-1}} [0,s(z)z]$.  

Let
\eq{
M = \max\limits_{z\in S}\|z_n\|; \quad m=\inf\limits_{z\in \partial S\setminus K_\delta} \Big(f(z)\Big); \quad \eps = \frac{m}{2M}.
}
Note that~$m > 0$ since~$\partial S \setminus K_\delta$ is a compact set, which does not intersect~$V$. We will shortly prove that with this choice of~$\eps$,~$A(y) > f(y)$ when~$y\in K_\delta \cup (U\setminus S)$; in such a case, we may set~$C_0$ (the compact set we are looking for) equal to~$S$ and~$A_0(y) := A(y) + \eps_1$ for sufficiently small~$\eps_1$. If~$y\in K_\delta$ and~$\|y\| > \nu$, then
\eq{
A(y) \leq - \eps\delta\nu < 0 \leq f(y).
} 
In the case~$y\notin K_\delta$, we also have~$y\notin S$. Let~$y^*$ be the point on the intersection of the segment~$[0,y]$ with the boundary of~$S$. Then,
\eq{
A(y) = -\eps y_n = - \eps \frac{\|y\|}{\|y^*\|}\,  y^*_n  \leq \frac{m}{2} \frac{\|y\|}{\|y^*\|} < \frac{\|y\|}{\|y^*\|} f(y^*) \leq f(y),
} 
the last inequality in the chain follows from the convexity of~$f$ (since~$y^* \in [0,y)$).
\end{proof}

\begin{cor}\label{Just convex approximation from inside}
Let~$\eps \colon U\to \R_+$ be a positive continuous function. Assume the graph of~$f$ does not contain lines. There exists a convex function~$g\colon U\to \R$ such that
\eq{
f(x) < g(x) \leq f(x) + \eps(x),\quad x\in U.
}
\end{cor}
\begin{proof}
Let the graph of~$g$ coincide with the convex hull of the graph of~$f+\eps$. The inequality~$g \leq f+\eps$ is evident. To prove the inequality~$f < g$, we need to modify the function~$A_{x}$ provided by Theorem~\ref{NoLinesTheorem}. Given any~$x\in U$, we will construct an affine function~$\tilde{A}_{x}$ such that~$f(x) < \tilde{A}_{x}(x)$ and~$\tilde{A}_{x}(y) < f(y) +\eps(y)$ for all~$y\in U$. This will prove the desired inequality~$f<g$.

Let~$C_x$ be the compact set delivered by Theorem~\ref{NoLinesTheorem} together with~$A_x$. We pick some number~$\theta \in (0,1)$ such that
\eq{
\theta < \frac{\inf \set{\eps(y)}{y \in C_x}}{\sup\set{A_{x}(y) - f(y)}{y\in C_x}}
}
and define
\eq{
\tilde{A}_{x}(y) = \theta A_{x}(y) + (1-\theta)\big(f(x) +L(y-x)\big),
}
here~$L\in\partial f(x)$ is an arbitrary function. Then,
\mlt{
\tilde{A}_{x}(y) - f(y) =\\ \theta\big(A_{x}(y) - f(y)\big) + (1-\theta)\big(f(x) + L(y-x) - f(y)\big) < \eps(y),\qquad y\in C_x.
} 
In the case~$y\notin C_x$ we simply have~$\tilde{A}_{x}(y) < f(y)$; the inequality~$f(x) < \tilde{A}_{x}(x)$ is also true.
\end{proof}

\section{Proofs of theorems \ref{main thm for lower approximation of convex functions} and \ref{main thm for upper approximation of convex functions}, and of Corollary \ref{corollary for approximation of convex functions}.}\label{S3}

{
We need to gather some facts and techniques from \cite{Azagra}. For instance we will be using {\em smooth maxima}: for any number $\delta>0$, denote
$$
M_{\delta}(x,y)=\frac{x+y+\theta(x-y)}{2}, \,\,\, (x, y)\in\R^2,
$$
where $\theta\colon\R\to (0,
\infty)$ is a $C^\infty$ function such that:
\begin{enumerate}
\item $\theta(t)=|t|$ if and only if $|t|\geq\delta$;
\item $\theta$ is convex and symmetric;
\item $\textrm{Lip}(\theta)=1$.
\end{enumerate}
If $f, g\colon U\subseteq\R^n\to\R$, define the function $M_{\delta}(f,g)\colon U\to\R$ by
$$
M_{\delta}(f,g)(x)=M_{\delta}(f(x), g(x)).
$$
By~$Lip(f)$ we mean the Lipschitz constant of~$f$.
\begin{prop}\label{properties of M(f,g)}
Let $f, g\colon U\subseteq\R^n\to\R$
be convex functions. For every $\delta>0$, the function
$M_{\delta}(f,g)\colon U\to\R$ has the following properties:
\begin{enumerate}
\item $M_{\delta}(f,g)$ is convex.
\item If $f$ is $C^k$ on $\{x\colon f(x)\geq g(x)-\delta\}$ and $g$ is $C^k$ on $\{x\colon g(x)\geq f(x)-\delta\}$ then $M_{\delta}(f,g)$ is $C^k$ on $U$. In particular, if $f, g$ are $C^k$, then so is $M_{\delta}(f,g)$.
\item $M_{\delta}(f,g)(x)=f(x)$ if $f(x)\geq g(x)+\delta$.
\item $M_{\delta}(f,g)(x)=g(x)$ if $g(x)\geq f(x)+\delta$.
\item $\max\{f,g\}\leq M_{\delta}(f,g)\leq \max\{f,g\} + \delta/2$.
\item $M_{\delta}(f,g)=M_{\delta}(g, f)$.
\item $\textrm{Lip}(M_{\delta}(f,g){|_B})\leq \max\{ \textrm{Lip}(f{|_B}), \textrm{Lip}(g{|_B}) \}$ for every ball $B\subset U$.
\item If $f, g$ are strictly convex on a set $B\subseteq U$, then so is $M_{\delta}(f,g)$.
\item If $f, g\in C^2(U)$ are strongly convex on a set $B\subseteq U$, then so is $M_{\delta}(f,g)$.
\item If $f_1\leq f_2$ and $g_1\leq g_2$ then $M_{\delta}(f_1, g_1)\leq M_{\delta}(f_2, g_2)$.
\end{enumerate}
\end{prop}
\begin{proof}
See \cite[Section 2]{Azagra}.
\end{proof}

The result below follows from the proof of \cite[Theorem 1.1]{Azagra}, although it was not explicitly mentioned there.
\begin{thm}\label{Uniform approximation by strongly convex functions}
Let $U$ be a nonempty convex open subset of $\R^n$, and $f\colon U\to\R$ be convex. Assume that $f$ cannot be written as $f=c\circ P+\ell$, where $P\colon \R^n\to\R^k$ is linear and surjective, $k<n$, $c\colon P(U)\to\R$, and $\ell$ is linear. Then $f$ can be uniformly approximated on $U$ by real-analytic strongly convex functions.
\end{thm}
For the sake of completeness, let us review the main points of the proof of \cite[Theorem 1.1]{Azagra} and make some remarks as to why the approximations can be taken strongly convex if $f$ is not of the form $f=c\circ P+\ell$. We will use some terminology from \cite{Azagra}.
\begin{defn} 
We will say that a function $C\colon\R^n\to\R$ is a
$k$-dimensional {\em corner function} on $\R^n$ if it is of the
form
    $$
C(x)=\max\{\, \ell_1 +b_1, \, \ell_2 +b_2, \, ..., \, \ell_k +b_k
\, \},
    $$
where the $\ell_j\colon\R^n\to\R$ are linear functions such that the
functions $L_{j}\colon\R^{n+1}\to\R$ defined by $L_{j}(x,
x_{n+1})=x_{n+1}-\ell_j(x)$, $1\leq j\leq k$, are linearly
independent, and the $b_j\in\R$. 
\end{defn}
We will also say that a convex
function $f\colon U\subseteq\R^n\to\R$ is supported by $C$ at a point $x\in U$
provided we have $C\leq f$ on $U$ and $C(x)=f(x)$.

The following lemma is a refinement of \cite[Lemma 4.2]{Azagra}.

\begin{lem}\label{reduction to Rk with k less than n}
Let $U\subseteq\R^n$ be open and convex, let~$f\colon U\to\R$ be a convex function and $x_0\in U$. Assume that $f$ is not supported at $x_0$ by any
$(n+1)$-dimensional corner function. Then there exist $k<n$, a
linear projection $P\colon\R^n\to\R^k$, a $C^p$ convex function
$c\colon P(U)\subseteq\R^k\to\R$, and a linear function $\ell\colon\R^n\to\R$ such that
$f=c\circ P+\ell$.
\end{lem}
In the statement of \cite[Lemma 4.2]{Azagra}, $f$ was assumed to be $C^1$, but this was just for convenience; the same proof can be used to show the result for an arbitrary convex function (using the fact that if the range of the subdifferential of a
convex function is contained in $\{0\}$ then the function is
constant, and applying this to the function $(t_1,...,
t_{n-k})\mapsto (f-\ell_1)(y+\sum_{j=1}^{n-k}t_j w_j)$).

\begin{proof}[Proof of Theorem~\ref{Uniform approximation by strongly convex functions}]
In order to show Theorem \ref{Uniform approximation by strongly convex functions}, one can argue as follows. Let us consider a compact convex subset $K$ of $U$. Given $\varepsilon>0$, since $f$ is convex and
Lipschitz on $K$ we can find finitely many points $x_1, ...., x_m\in K$ and affine functions $h_1,
..., h_m\colon\R^n\to\R$ such that $f$ is differentiable at each $x_j$, each $h_j$ supports $f-\varepsilon$ at $x_j$, and $f-2\varepsilon\leq \max\{h_1, ...,
h_m\}$ on $K$. By convexity we also have $\max\{h_1, ..., h_{m}\}\leq f-\varepsilon$
on all of $U$. By the preceding lemma, for each $x_j$ we may find a $(n+1)$-dimensional corner function $C_j$ that supports $f-\varepsilon$ at $x_j$. Note that these corner functions are always defined on all of $\R^n$ (even when $f$ is not). Since $f$ is
convex and differentiable at $x_j$, we have $h_j=C_j$ on a neighborhood of
$x_j$ and, by convexity, also $h_j\leq C_j\leq f-\varepsilon$ and
$\max\{C_1, ..., C_m\}\leq f-\varepsilon$ on $U$. We also have
$f-2\varepsilon\leq \max\{h_1,..., h_m\}\leq\max\{C_1, ...,
C_{m}\}\leq f-\varepsilon$ on $K$.  Now apply \cite[Lemma 4.1]{Azagra} to the functions $C_j+\varepsilon'/2$ in order to find $C^{\infty}$ strongly
convex functions $g_{1}, ..., g_{m}\colon \R^n\to\R$ such that $C_j\leq
g_j\leq C_j +\varepsilon'$, where $\varepsilon':=\varepsilon/2m$,
and define $g\colon\R^n\to\R$ by
$$g=M_{\varepsilon'}(g_1, M_{\varepsilon'}(g_2,
M_{\varepsilon'}(g_3, ..., M_{\varepsilon'}(g_{m-1}, g_m))...))$$
(for instance, if $m=3$, then $g=M_{\varepsilon'}(g_1,
M_{\varepsilon'}(g_2, g_3))$). By Proposition \ref{properties of
M(f,g)}, we have that $g\in C^{\infty}(\R^d)$ is strongly convex,
    $$
\max\{C_1, ..., C_m\}\leq g\leq \max\{C_1, ...,
C_m\}+m\varepsilon'\leq f-\frac{\varepsilon}{2} \,\,\, \textrm{ on
} \,\,\, U,
    $$
and
    $$
f-2\varepsilon\leq \max\{C_1, ..., C_m\}\leq g \,\,\, \textrm{ on
} \,\,\, K.
    $$
Therefore, $f\colon U\subseteq\R\to\R$ can be approximated from below by $C^\infty$
strongly convex functions, uniformly on each compact convex subset of
$U$. By \cite[Theorem 1.2]{Azagra} and Remark 1 in Section 2 of the same paper, we conclude that, given
$\varepsilon>0$ we may find a $C^\infty$ strongly convex function
$h$ such that $f-2\varepsilon\leq h\leq f-\varepsilon$ on $U$.

Finally, set 
$$
\eta(x)=\frac{1}{2}\min\{ \varepsilon, \, \min\{D^{2}h(x)(v)^2 \colon v\in \R^{n}, \|v\|=1\}\}, \,\,\, x\in U.
    $$
The function $\eta\colon U\to (0, \infty)$ is
continuous, so we can apply Whitney's theorem (Lemma~$6$ in~\cite{Whitney}) on $C^2$-fine
approximation of $C^2$ functions by real-analytic functions to find a real analytic function $g\colon U\to\R$ such
that
    $$
\max\{|h-g|, \|Dh-Dg|, \|D^2h-D^2g|\}\leq \eta.
    $$
This implies that $f-3\varepsilon\leq g\leq f$ and
$D^2g\geq\frac{1}{2}D^2h>0$, so $g$ is strongly convex as well.
\end{proof}

We will also make use of the following simple fact.

\begin{lem}\label{no lines lemma}
Let $I=(a,b)$, where $-\infty\leq a < b\leq+\infty$, let~$\varphi\colon I\to\R$ be convex and let~$\psi\colon\R\to [0, 1]$ be differentiable and such that $\lim_{|t|\to\infty}\psi(t)=0$. If $I\neq\R$, also assume that $\psi^{-1}(0)=\R\setminus I$. If $|\varphi(t)|\leq \psi(t)$ for all $t\in I$, then $\varphi(t)=0$ for all $t\in I$.
\end{lem}
\begin{proof}
If $\varphi(s)>0$ for some $s\in I$, then $\varphi$ attains a maximum in $I$, and since $\varphi$ is convex and $\lim_{t\to a^{+}}\varphi(t)=0$, $\varphi$ it must be constantly $0$. Hence $\varphi\leq 0$. Let us see that $\varphi(t)=0$ for all $t\in I$. Take $t_0\in I$. If $a\in\R$, by convexity we have
$$
-\frac{\psi(t)}{t-a}\leq \frac{\varphi(t)}{t-a}\leq \frac{\varphi(t_0)}{t_0-a} \textrm{ for all } t\in [a, t_0],
$$
but
$$
\lim_{t\to a^{+}}\frac{-\psi(t)}{t-a}=-\psi'(a)=0,
$$
so $0\leq \varphi(t_0)$. If $b\in\R$, similarly we get $\varphi(t_0)\geq 0$. Finally, if $I=\R$, since $\varphi\leq 0$ is convex, $\varphi$ must be constant, and the assumptions that $\lim_{|t|\to\infty}\psi(t)=0$ and $|\varphi|\leq\psi$ imply that this constant must be $0$.
\end{proof}

\medskip

\begin{proof}[Proof of Theorem \ref{main thm for upper approximation of convex functions}]
\noindent $(2)\implies (3)$ is trivial.

\medskip
\noindent $(3)\implies (1)$ is a consequence of Lemma \ref{no lines lemma}: if the restriction of $f$ to $U\cap\{x+tv\colon t\in\R\}$ is affine, we may consider a function $\varepsilon\colon\R^d\to [0, 1]$ of class $C^1$ such that $\lim_{|x|\to\infty}\varepsilon(x)=0$ and $\R^n\setminus U=\varepsilon^{-1}(0)$ (if $U\neq\R^n$). By assumption there exists a convex function $g\colon U\to\R$ such that $f<g<f+\varepsilon$. Then we may apply Lemma \ref{no lines lemma} with $\varphi(t):=g(x+tv)-f(x+tv)$ and $\psi(t)=\varepsilon(x+tv)$ to find that $g(x+tv)=f(x+tv)$ for all $t$, contradicting that $f<g$.

\medskip

\noindent $(1)\implies (2)$:
We may assume that $\lim_{|x|\to\infty}\varepsilon(x)=0$, and if $U\neq\R^n$, we may also assume that $\varepsilon$ has $C^1$ extension to all of $\R^n$, denoted still by $\varepsilon$, such that $\varepsilon^{-1}(0)=\R^n\setminus U$. Let us fix a sequence of compact sets $(K_j)$ such that
$$
U=\bigcup_{j=1}^{\infty}K_j \,\,\, \textrm{ and} \,\,\, K_j \subset \textrm{int}(K_{j+1}) \textrm{ for every  } j\in \N.
$$
By Corollary \ref{Just convex approximation from inside} there exists a convex function $h_1\colon U\to\R$ such that
$$
f<h_1<f+\varepsilon.
$$
Since the graph of $f$ contains no lines, using the preceding lemma it is easy to see that the graph of $h_1$ contains no lines either (if the restriction of $f$ to $U\cap\{x+te\colon t\in\R\}$ is affine, we may apply the lemma with the functions $\varphi(t):=h_1(x+te)-f(x+te)$ and $\psi(t):=\varepsilon(x+te)$). In particular $h_1$ is not of the form $h_1=c\circ P+\ell$ for any linear projection $P\colon \R^n\to\R^k$ with $k<n$ and $\ell$ linear. Then, according to Theorem \ref{Uniform approximation by strongly convex functions}, we may find a strongly convex $C^{\infty}$ function $g_1\colon U\to\R$ such that
$$
h_1-\frac{\varepsilon_1}{3}<g_1<h_1-\frac{\varepsilon_1}{6} \textrm{ on } U,
$$
where 
$$
\varepsilon_1 :=\inf_{x\in K_1}\{h_1(x)-f(x)\}.
$$
For future notational consistency, we also write~$\varphi_1=g_1$.
By continuity of $\varepsilon$ and compactness of $K_1$ there exists $m_1\in\N$ such that
$$
f+\frac{\varepsilon}{m_1}<h_1-\frac{2}{3}\varepsilon_1 \textrm{ on } K_1,
$$
and applying again Corollary \ref{Just convex approximation from inside} we can take a convex function $h_2\colon U\to\R$ such that
$$
f<h_2<f+\frac{\varepsilon}{m_1} \textrm{ on } U.
$$
Observe that the graph of $h_2$ cannot contain any line. Now let us set
$$
\varepsilon_2 :=\inf_{x\in K_2}\{h_2(x)-f(x)\}>0,
$$
and use Theorem \ref{Uniform approximation by strongly convex functions} to obtain a strongly convex $C^{\infty}$ function $\varphi_2\colon U\to\R$ such that
$$
h_2-\frac{\varepsilon_2}{3}<\varphi_2<h_2-\frac{\varepsilon_2}{6} \textrm{ on } U.
$$
Let us define
$$
g_2 :=M_{\delta_2}(g_1, \varphi_2),
$$
where $\delta_2=\varepsilon_2/12$. By using Proposition \ref{properties of M(f,g)} we see that $g_2$ is a strongly convex $C^\infty$ function satisfying
$$
\max\{g_1, \varphi_2\}\leq g_2\leq \max\{g_1, \varphi_2\} +\delta_2/2.
$$
Also, since $\varphi_2<h_2 -\varepsilon_2/6 <h_1-2\varepsilon_1/3<g_1-\varepsilon_1/3$ on $K_1$, and $\varepsilon_1/3>\delta_2$, we obtain
$$
g_2=g_1 \textrm{ on } K_1.
$$
Moreover, we have
$$
f<g_2<f+\varepsilon \textrm{ on } K_2,
$$
because
$$
g_2\geq\varphi_2>h_2-\varepsilon_2/3>h_2-\varepsilon_2\geq f \textrm{ on } K_2
$$
and
\mltt{
g_2 \leq \max\{g_1, \varphi_2\}+\delta_2/2 \leq\max\{h_1-\varepsilon_1/6, h_2-\varepsilon_2/6\}+\delta_2/2\leq\\ \max\{h_1, h_2\}-\varepsilon_2/12 <f+\varepsilon 
}
on $U$.


We continue this process by induction: suppose that, for $N\geq 2$, we have defined convex functions $h_1, ..., h_N\colon U\to\R$, strongly convex functions $g_1, ..., g_N$, and $\varphi_1, ..., \varphi_N\in C^{\infty}(U)$ (with $\varphi_1=g_1$),
numbers $1=m_0<m_1<m_2< ...< m_{N-1}\in\N$ such that, for every $j=1, ..., N$,
$$
f+\frac{\varepsilon}{m_{j}}<h_j-\frac{2}{3}\varepsilon_j \textrm{ on } K_j,
$$
and
$$
h_j-\frac{\varepsilon_j}{3}<\varphi_j<h_j-\frac{\varepsilon_j}{6} \textrm{ on } U,
$$
where
$$
\varepsilon_j :=\inf_{x\in K_j}\{h_j(x)-f(x)\},
$$
$$
g_j=M_{\delta_{j}}(g_{j-1}, \varphi_j),
$$
with
$$
\delta_j=\frac{\varepsilon_j}{3\cdot 2^{j}},
$$
and also
$$
g_j=g_{j-1} \textrm{ on } K_{j-1},
$$
$$
g_{j}\leq\max\{h_1, ..., h_{j}\}-\delta_j <f+\varepsilon \textrm{ on } U,
$$
$$
g_j\geq h_j-\varepsilon_j\geq f \textrm{ on } K_j.
$$
Then we can find $m_N\in\N$ such that  $m_N>m_{N-1}$ and
$$
f+\frac{\varepsilon}{m_N}<h_{N}-\frac{2}{3}\varepsilon_N \textrm{ on } K_N,
$$
and using Corollary \ref{Just convex approximation from inside}, we obtain a convex function $h_{N+1}\colon U\to\R$ such that
$$
f<h_{N+1}<f+\frac{\varepsilon}{m_N} \textrm{ on } U.
$$
According to Lemma \ref{no lines lemma} the graph of $h_N$ cannot contain any line, so by Theorem \ref{Uniform approximation by strongly convex functions}, for
$$
\varepsilon_{N+1}:=  \inf_{x\in K_{N+1}}\{h_{N+1}(x)-f(x)\}>0,
$$
there exists a strongly convex function $\varphi_{N+1}\in C^{\infty}(U)$ such that
$$
h_{N+1}-\frac{\varepsilon_{N+1}}{3}<\varphi_{N+1}<h_{N+1}-\frac{\varepsilon_{N+1}}{6} \textrm{ on } U.
$$
We define $$\delta_{N+1}=\frac{\varepsilon_{N+1}}{3\cdot 2^{N+1}},$$ and
$$
g_{N+1}=M_{\delta_{N+1}}(g_N, \varphi_{N+1}),
$$
which is a strongly convex $C^\infty$ function satisfying
$$
\max\{g_N, \varphi_{N+1}\}\leq g_{N+1}\leq \max\{g_N, \varphi_{N+1}\} +\delta_{N+1}/2.
$$
Since $\varphi_{N+1}<h_{N+1} -\varepsilon_{N+1}/6 <h_{N+1}-2\varepsilon_N/3<g_N-\varepsilon_N/3$ on $K_N$, and $\varepsilon_N/3>\delta_{N+1}$, Proposition \ref{properties of M(f,g)} implies
$$
g_{N+1}=g_N \textrm{ on } K_N.
$$
On the other hand,
$$
g_{N+1}\geq\varphi_{N+1}>h_{N+1}-\varepsilon_{N+1}/3>h_{N+1}-\varepsilon_{N+1}\geq f \textrm{ on } K_{N+1}
$$
and
\begin{eqnarray*}
& & g_{N+1} \leq \max\{g_N, \varphi_{N+1}\}+\frac{\delta_{N+1}}{2}\leq \max\{
\max\{h_1, ..., h_{N}\}-\delta_N, \, h_{N+1}-\frac{\varepsilon_{N+1}}{6}\}+\delta_{N+1} \\
& & \leq  \max\{h_1, ..., h_{N+1}\}-\delta_{N+1}<f+\varepsilon
\end{eqnarray*}
on $U$.

Therefore, by induction there exist sequences of functions $(g_j)$, $(\varphi_j)$, $(h_j)$ satisfying the properties listed above for every $j\in\N$.

Let us finally define
$$
g(x)=\lim_{j\to\infty}g_j(x), \,\,\, x\in U.
$$
Since the $g_j\in C^{\infty}(U)$ are strongly convex and satisfy $g_{j+1}=g_{j}$ on $K_j$,  $K_j\subset \mathrm{int}\, K_{j+1}$  for every $j$, and $U=\bigcup_{j\in\N} K_j$, it is clear that $g$ is well defined, strongly convex, and of class $C^{\infty}(U)$. We also have $g_j>f$ on $K_j$ for every $j$, and $g_j<f+\varepsilon$ on $U$ for every $j$, which imply $f<g<f+\varepsilon$ on $U$.

In order to obtain a real-analytic function $\psi$ with these properties, let 
\mltt{
\eta(x) :=\\ \frac{1}{2}\min\Big\{ g(x)-f(x), \, f(x)+\varepsilon(x)-g(x), \, \min\big\{D^{2}g(x)(v)^2 \,
\colon v\in \R^{n}, \|v\|=1\big\}\Big\}, \\  x\in U,
   }
which defines a strictly positive continuous function on $U$. We can apply Whitney's theorem (Lemma~$6$ in~\cite{Whitney}) on $C^2$-fine
approximation of $C^2$ functions to find a real-analytic function $\psi\colon U\to\R$ such
that
    $$
\max\{|\psi-g|, \|D\psi-Dg|, \|D^2\psi-D^2g|\}\leq \eta.
    $$
This implies that $f<\psi<f+\varepsilon$ and
$D^2\psi\geq\frac{1}{2}D^2g>0$, so $g$ is strongly convex too. 
\end{proof}

\medskip

\begin{proof}[Proof of Theorem \ref{main thm for lower approximation of convex functions}]

\noindent $(1)\implies (2)$: Since the graph of $f$ does not contain any ray, it does not contain any line either. Then, according to Theorem \ref{NoRaysTheorem}, there exists a convex function $h\colon U\to\R$ such that $f-\varepsilon<h<f$. Setting $\delta(x)=f(x)-h(x)$, $x\in U$, we may apply Theorem \ref{main thm for upper approximation of convex functions}  to $h$ to find a real-analytic strongly convex function $g\colon U\to\R$ such that $h<g<h+\delta$, which implies $f-\varepsilon<g<f$.

\noindent $(2)\implies (3)$ is trivial.

\noindent $(3)\implies (1)$ can be proved by using the following variant of Lemma \ref{no lines lemma} with the function $\varphi(t)=g(x+tv)-f(x+tv)$, $t\in (a, t_0]$, assuming that the graph of $f$ is affine on some ray $\{x+tv : t\in (a, t_0]\}$ of $U$.
\begin{lem}\label{no rays lemma}
Let $I=(a,b)$, where $-\infty\leq a < b\leq+\infty$, let~$\varphi\colon I\to\R$ be convex and let $\psi\colon \R\to [0, 1]$ be differentiable and such that $\lim_{|t|\to\infty}\psi(t)=0$. If $I\neq\R$, also assume that $\psi^{-1}(0)=\R\setminus I$. Let $t_0\in I$. If $-\psi(t)\leq\varphi(t)$ for all $t\in I\cap (a, t_0]$, then $\varphi(t)\geq 0$ for all $t\in I$ sufficiently close to~$a$.
\end{lem}
The proof of this lemma is similar to that of Lemma \ref{no lines lemma} and is left to the reader.
\end{proof}

\begin{proof}[Proof of Corollary \ref{corollary for approximation of convex functions}]
$(1)\implies (2)$ is an obvious consequence of Theorem \ref{main thm for upper approximation of convex functions}, and $(2)\implies (3)$ is trivial. Let us see that $(3)\implies (1)$: assume $(1)$ is false; then there exists a line $\{x+tv\colon t\in\R\}\cap U=\{x+tv\colon t\in (a, b)\}$, on which $f$ is affine. Let $\varepsilon\colon \R^n\to [0, 1]$ be of class $C^1$ and such that $\lim_{|x|\to\infty}\varepsilon(x)=0$ and (if $U\neq\R^n$) also $\R^n\setminus U=\varepsilon^{-1}(0)$. By the assumption there exists a strictly convex function $g\colon U\to\R$ such that $|f-g|<\varepsilon$. Then, by applying Lemma \ref{no lines lemma} with $\varphi(t):=g(x+tv)-f(x+tv)$ and $\psi(t)=\varepsilon(x+tv)$, we deduce that $g(x+tv)=f(x+tv)$ for all $t$. This contradicts that $g$ is strictly convex.
\end{proof}

\section{Proofs of Theorems \ref{main thm on outer approximation} and \ref{main thm on inner approximation}.}\label{S4}

Besides Theorems \ref{main thm for lower approximation of convex functions} and \ref{main thm for upper approximation of convex functions}, in the proofs of Theorems \ref{main thm on outer approximation} and \ref{main thm on inner approximation} we will use the following lemmas.

\begin{lem}[See \cite{AzagraHajlasz}, Lemma 3.2]\label{mu is essentially coercive}
Let $W\subset \R^n$ be a closed convex set such that $0\in\mathrm{int}(W)$, and let $\mu=\mu_{W}$ denote the Minkowski functional of $W$. The following assertions are equivalent:
\begin{enumerate}
\item[{(a)}] $W$ does not contain any line. 
\item[{(b)}] $\partial W$ does not contain any line.
\item[{(c)}] $\mu^{-1}(0)$ does not contain any line
\item[{(d)}] $\mu$ is essentially coercive.
\end{enumerate}
\end{lem}

\begin{lem}\label{If S does not contain rays then S is a graph}
Let $S=\partial V$, where $V$ is a closed convex set $V$ with nonempty interior in $\R^n$. If $S$ does not contain any rays and is unbounded, then $S$ can be regarded (up to a suitable rotation and translation) as the graph of a convex function $f:U\subseteq\R^{n-1}\to\R$ such that $\lim_{y\in U, |y|\to\infty}f(y)=\infty$ (if $U$ is unbounded) and $\lim_{y\to x}f(y)=\infty$ for every $x\in\partial U$ (if $U\neq\R^{n-1}$). In particular $f$ is properly convex and its graph contains no ray.
\end{lem}
\begin{proof}
Since $S$ does not contain any line, nor does $V$ (according to the preceding lemma). And since $S$ is unbounded, so is $V$, hence $V$ contains a ray. Consider the maximal inscribed cone of~$V$:
\eq{
K = \Set{e\in\R^n}{\exists x\in U \text{ such that } \set{x+te}{t > 0} \subset V}.
}
The cone~$K$ is non-empty, closed, convex, and does not contain lines. Consider the hyperplane~$H$ constructed in the proof of Theorem~\ref{SeparationTheorem} (we need the explicit construction with the closest point~$\zeta$ presented in the proof). Let us introduce the orthogonal coordinates  such that~$H = \set{x\in \R^n}{x_n=0}$ and~$x_n > 0$ on~$K$. Note that in such a case~$\zeta$ lies on the~$Ox_n$ axis, which yields the ray~$(0,0,\ldots,0,t)$, where~$t\in \R_+$, belongs to~$K$. We will call this ray the positive half of the~$Ox_d$-axis.

Let~$P_H$ be the orthogonal projection of~$\R^n$ onto~$H$. Set~$U = P_H(V)$ and~$f(y) = \inf \set{t\in\R}{(y,t)\in V}$, here~$y\in U$. Let us prove that this choice indeed fulfills the requirements. First, since the positive half of the~$Ox_d$ axis lies in~$K$, any ray of the form~$\set{x+(0,0,\ldots,0,t)}{t > 0}$ lies in~$V$, provided~$x\in V$. Thus,~$V$ is indeed the epigraph of~$f$. 

Second, let us check two limit assertions. Similar to~\eqref{StrictPositivity},
\eq{
K \subset \Set{y\in \R^n}{y_n > \delta \|y\|}
}
for a sufficiently small number~$\delta>0$. This, in particular, leads to the bound~$f(y) \geq \delta \|y\| - C$ for a sufficiently large constant~$C$. Therefore,~$f(x)$ tends to infinity as~$x$ tends to infinity inside~$U$, supporting the first limit assertion. On the other hand, if $U\neq\R^{n-1}$ and $x\in\partial U$, the limit $\beta:=\lim_{y\to x}f(y)$ exists in $(0, +\infty]$. If $\beta$ were finite then $S$ would contain the ray $\{(x, t) \colon t\geq\beta\}$, contradicting the assumption that $S$ contain no rays. Therefore $\beta$ is infinite, and the second limit assertion is also verified. This also shows~$S$ is the graph of~$f$.
\end{proof}

\begin{proof}[Proof of Theorem \ref{main thm on outer approximation}]
$(1)\implies (3)$: If $S$ is unbounded then this implication is an immediate consequence of Theorem \ref{main thm for lower approximation of convex functions} and Lemma \ref{If S does not contain rays then S is a graph}. On the other hand, if $S$ is compact, the result is well known. Nevertheless, for completeness, let us provide a short proof of this case based on the preceding results. We may assume that $0\in\textrm{int}(V)$ and $U$ is of the form $\varphi^{-1}(1-2\varepsilon, 1+2\varepsilon)$, where $\varphi=\mu^2$, $\mu$ denoting the Minkowski functional of $V$, and $\varepsilon$ is a positive constant. The function $\varphi$ is convex and coercive, and its graph does not contain any ray. By Theorem \ref{main thm for lower approximation of convex functions} there exists a real-analytic strongly convex function $g\colon\R^n\to\R$ such that $\varphi-\varepsilon<g<\varphi$. Let us define $W=g^{-1}(-\infty, 1]$. Then $S_U:=\partial W=g^{-1}(1)$ is a real-analytic strongly convex hypersurface with $S_U\subset U\setminus V$.

\noindent $(3)\implies (2)$ is obvious. 

\noindent $(2)\implies (1)$: this can be proved similarly to $(3)\implies (1)$ of Theorem \ref{main thm for lower approximation of convex functions}. The details are left to the reader.
\end{proof}

\begin{proof}[Proof of Theorem \ref{main thm on inner approximation}]

\noindent $(1)\implies (3)$: By Lemma \ref{mu is essentially coercive}, the Minkowski functional of $V$, which we will denote $\mu$, is essentially coercive (and in particular its graph does not contain any line). Given an open set $U\supset S$, by using partitions of unity for instance, it is not difficult to construct a continuous function $\varepsilon\colon\R^n\to (0,1]$ such that
$$
\mu(x)+\varepsilon(x)<1 \textrm{ for all } x\in V\setminus U.
$$
Then we may apply Theorem \ref{main thm for upper approximation of convex functions} to find a real-analytic strongly convex function $g\colon\R^n\to\R$ such that $\mu<g<\mu+\varepsilon$. Define $S_U=g^{-1}(1)$. It is clear that $S_U$ is a real-analytic strongly convex hypersurface (the boundary of the convex body $g^{-1}(-\infty, 1]$). If $x\in V\setminus U$, then we have $g(x)<\mu(x)+\varepsilon(x)<1$, and if $x\in\R^n\setminus \textrm{int}(V)$, then $g(x)>\mu(x)\geq 1$. Therefore, $S_U=g^{-1}(1)\subset U\cap \textrm{int}(V)$.

\noindent $(3)\implies (2)$ is trivial.

\noindent $(2)\implies (1)$ is similar to $(3)\implies (1)$ of Theorem \ref{main thm for upper approximation of convex functions}. The details are left to the reader.
\end{proof}

\section{Proof of Theorem \ref{C1 fine result}.}\label{S5}

Let us gather some preliminary results that we will be using in the proof. The following theorem is well known; see, for instance, \cite[Theorem 25.7]{Rockafellar}).

\begin{thm}\label{automatic convergence of gradients}
Let $U$ be a nonempty open convex subset of $\R^n$, let $f:U\to\R$ be a differentiable convex function, and $(f_k)$ be a sequence of differentiable convex functions such that $f(x)=\lim_{k\to\infty}f_k(x)$ for every $x\in U$. Then $Df_k$ converges to $Df$, uniformly on each compact subset of $U$.
\end{thm}

The following fact about smooth maxima is shown in \cite[Lemma 7.1]{Azagra}.
\begin{lem}\label{estimate of first derivatives of smooth maxima}
Let $M_\delta$ the smooth maximum of Proposition \ref{properties of M(f,g)}, and let $V\subseteq\R^n$ be an open set.
If $\varphi, \psi\in C^{1}(V)$, then
$$
\| DM_{\delta}(\varphi, \psi)-\frac{D\varphi +D\psi}{2}\|\leq\frac{1}{2}\|D\varphi-D\psi\|.
$$
\end{lem}

We will also use the following consequence of Theorems \ref{NoLinesTheorem}, \ref{main thm for lower approximation of convex functions} and \ref{automatic convergence of gradients}. 

\begin{lem}\label{C1approxlemma}
Let $U$ be a nonempty open convex subset of $\R^n$, and $f:U\to\R$ be convex and $C^1$. Assume that the graph of $f$ contains no lines. Then, for every continuous function $\varepsilon:U\to (0,1)$ and every compact set $K\subset U$ there exist a compact set $C$ and a convex $C^1$ function $\psi:U\to\R$ such that:
\begin{enumerate}
\item $f\leq \psi<f+\varepsilon$ on $U$;
\item $f<\psi$ on $K$;
\item $K\subset \textrm{int}(C)\subset C \subset U$;
\item $f=\psi$ on $U\setminus C$;
\item $\psi$ is strongly convex on $\textrm{int}(C)$, and
\item $\|D\psi-Df\|<\varepsilon$ on $U$.
\end{enumerate}
\end{lem}
\begin{proof}
For every $x\in K$, by Theorem \ref{NoLinesTheorem} there exist an affine function $A_x:\R^n\to\R$ and a compact set $C_x\subset U$ such that
$f(x)-A_x(x)<0$ and $f(y)-A_x(y)>0$ for all $y\in U\setminus C_x$. In particular $D_x:=\{y\in U : f(y)-A_x(y)\leq 0\}$ is a compact convex neighborhood of $x$. Since $K$ is compact, we may find finitely many points $x_1, ..., x_m\in K$ such that
$$
K\subset \bigcup_{j=1}^{m}\textrm{int}(D_{x_j}).
$$
Observe that the graph of the restriction of $f$ to $\textrm{int}(D_{x_j})$ cannot contain any line for any $j=1, ..., m$. For every $j=1, ..., m$, let $\varepsilon_j:U\to [0,1]$ be a $C^1$ function such that $\varepsilon_{j}^{-1}(0)=U\setminus\textrm{int}(D_{x_j})$ and $\varepsilon_j\leq\varepsilon$. According to Theorem \ref{main thm for lower approximation of convex functions}, for each $j$ there exists a strongly convex $C^{\infty}$ function $\varphi_j:\textrm{int}(D_{x_j})\to\R$ such that $f<\varphi_j<\varepsilon_j$ on $\textrm{int}(D_{x_j})$. For each $j$, we can extend $\varphi_j$ to all of $U$ by setting $\varphi_j=f$ on $U\setminus\textrm{int}(D_{x_j})$, and since $\varepsilon_j$ is of class $C^1$ and satisfies $\varepsilon_j=0$ on $U\setminus\textrm{int}(D_{x_j})$, we have that $\varphi_j$ is differentiable on $U$, which (because $\varphi_j$ is convex) amounts to saying that $\varphi_j\in C^1(U)$.

Let us call $C=\bigcup_{j=1}^{m}D_{x_j}$ and
$$
\varphi=\frac{1}{m}\sum_{j=1}^{m}\varphi_j.
$$
It is easy to check that $C$ and $\varphi$ satisfy properties $(1)-(5)$ of the statement (with $\varphi$ in place of $\psi$). 

Now, for each $k\in\N$, we may apply what we have just proved with $\varepsilon/k$ replacing $\varepsilon$, and we obtain a sequence $(\psi_k)$ of $C^1$ convex functions satisfying properties $(1)-(5)$ (with the same $C$) and also
$$
f\leq \psi_k\leq f+\frac{\varepsilon}{k} \textrm{ on } U
$$
for every $k$. Then, by Theorem \ref{automatic convergence of gradients}, $D\psi_k$ converges to $Df$ uniformly on $C$, and therefore we can find $k$ large enough so that
$$
\|D\psi_k(x)-Df(x)\|\leq \min_{y\in C}\varepsilon(y) \textrm{ for all } x\in C.
$$
Since $\psi_k=f$ on $U\setminus C$, we also have $D\psi_k=Df$ on $U\setminus C$, so by setting $\psi=\psi_k$ we get a $C^1$ convex function satisfying properties $(1)-(6)$.
\end{proof}

We are ready to establish a $C^1$-fine version of Corollary \ref{Just convex approximation from inside}.

\begin{prop}\label{C1 version of just convex approx from inside}
Let $U$ be an nonempty open convex subset of $\R^n$, $f:U\to\R$ a convex function, and $\varepsilon\colon U\to (0,1)$ a continuous function. Assume that the graph of $f$ contains no lines. Then there exists a $C^1$ strongly convex function $g: U\to \R$ such that
$$
f<g<f+\varepsilon \textrm{ and } 
\|Dg-Df\|< \varepsilon \textrm{ on } U.
$$
\end{prop}
\begin{proof}
Let us fix a sequence of compact sets $(K_j)$ such that
$$
U=\bigcup_{j=1}^{\infty}K_j \,\,\, \textrm{ and} \,\,\, K_j \subset \textrm{int}(K_{j+1}) \textrm{ for every  } j\in \N.
$$
For every $j\in\N$, by Lemma \ref{C1approxlemma} there exist a compact set $C_j$ and a convex $C^1$ function $g_j:U\to\R$ such that:
\begin{enumerate}
\item $f\leq g_j\leq f+\varepsilon/2$ on $U$;
\item $f< g_j$ on $K_j$;
\item $K_j\subset \textrm{int}(C_j)\subset C_j \subset U$;
\item $f=g_j$ on $U\setminus C_j$;
\item $g_j$ is strongly convex on $\textrm{int}(C_j)$, and
\item $\|Dg_j-Df\|\leq \varepsilon/2$ on $U$.
\end{enumerate}
Let us define $$g=\sum_{j=1}^{\infty}\frac{1}{2^j}g_j.$$ It is routine to check that $g:U\to\R$ is of class $C^1$, strongly convex, and satisfies $f<g<f+\varepsilon$ and $\|Dg-Df\|< \varepsilon$ on $U$.
\end{proof}

Now let us proceed with the proof of Theorem \ref{C1 fine result}.
We only need to show that $(1)\implies (2)$ (the converse is easily shown as in the proof of Corollary \ref{corollary for approximation of convex functions}).

As in the proof of Theorem \ref{main thm for upper approximation of convex functions}, we may assume that $\lim_{|x|\to\infty}\varepsilon(x)=0$ and, if $U\neq\R^d$, that $\varepsilon$ has $C^1$ extension to all of $\R^d$, denoted still by $\varepsilon$, such that $\varepsilon^{-1}(0)=\R^d\setminus U$. Let us fix a sequence of compact sets $(K_j)$ such that
$$
U=\bigcup_{j=1}^{\infty}K_j \,\,\, \textrm{ and} \,\,\, K_j \subset \textrm{int}(K_{j+1}) \textrm{ for every  } j\in \N.
$$
By the preceding proposition there exist a compact set $C_1$ and a strongly convex $C^1$ function $h_1:U\to\R$ such that $h_1:U\to\R$ such that
$$
f<h_1<f+\varepsilon \textrm{ and } \|Dh_1-Df\|<\varepsilon \textrm{ on } U.
$$
Let us set
$$
\varepsilon_1 :=\inf_{x\in K_1}\{h_1(x)-f(x)\}.
$$
By continuity of $\varepsilon$ and compactness of $K_1$ there exists $m_1\in\N$ such that
$$
f+\frac{\varepsilon}{m_1}<h_1-\frac{2}{3}\varepsilon_1 \textrm{ on } K_1,
$$
and applying again Proposition \ref{C1 version of just convex approx from inside} we can take a $C^1$ strongly convex function $h_2\colon U\to\R$ such that
$$
f<h_2<f+\frac{\varepsilon}{m_1} \textrm{ and } \|Df- Dh_2\|\leq \frac{1}{4}\varepsilon \textrm{ on } U.
$$
Using the limiting properties of $\varepsilon$, the inequalities $f<h_1<f+\varepsilon$ and $f<h_2<f+\varepsilon/m_1$, and the fact that $U=\bigcup_{j=1}^{\infty} K_j$, we may find a number $n_2\in\N$ such that
$$
h_2>h_1-\frac{\varepsilon_{1}}{12} \textrm{ on } U\setminus K_{n_2}.
$$
We set $n_1=1$, 
$
\varepsilon_2=\inf_{y\in K_{n_2}}\{h_2(y)-f(y)\},
$
and find $n_3>n_2$ so that
$$
h_3>h_2-\frac{\varepsilon_2}{96} \textrm{ on } U.
$$
By continuing this process by induction, we obtain sequences $m_0=1<m_1<m_2<...$ and $n_1=1<n_2<n_3<...$ of positive integers, and $C^1$ strongly convex functions $h_j:U\to\R$, $j\in\N$, satisfying
$$
f+\frac{\varepsilon}{m_{j}}<h_j-\frac{2}{3}\varepsilon_j \textrm{ on } K_{n_j},
$$
$$
f<h_{j}<f+\frac{\varepsilon}{m_{j-1}} \textrm{ and } \|Dh_j-Df\|<\frac{1}{4}\varepsilon \textrm{ on } U,
$$
and
$$
h_{j+1}>h_j-\frac{\varepsilon_j}{2^{3j-1}\cdot 3} \textrm{ on } U\setminus K_{n_{j+1}},
$$
where
$$
\varepsilon_j :=\inf_{x\in K_{n_j}}\{h_j(x)-f(x)\}.
$$
Next, for every $j\in\N$, we may combine Theorems \ref{main thm for upper approximation of convex functions} and \ref{automatic convergence of gradients} in order to find a $C^{\infty}$ strongly convex function $\varphi_j:U\to\R$ such that
$$
h_j-\frac{\varepsilon_j}{2^{3j-3}\cdot 3}<\varphi_j<h_j-\frac{\varepsilon_j}{2^{3j-2}\cdot 3} \textrm{ on } U,
$$
and
$$
\|D\varphi_j-Dh_j\|<\frac{1}{4}\inf\{ \varepsilon(y) : y\in K_{n_{j+1}} \} \textrm{ on } K_{n_{j+1}}.
$$
Now let us call $\varphi_1=g_1$, and for every $j\geq 2$, define
$$
g_j=M_{\delta_j}\left( g_{j-1}, \varphi_j\right),
$$
where
$$
\delta_j=\frac{\varepsilon_j}{2^{3j-2}\cdot 3}.
$$
\begin{claim}\label{estimates on g}
For every $j\in\N$, $g_j$ is a $C^{\infty}$ strongly convex function satisfying:
\begin{enumerate}
\item $g_{j+1}=g_j$ on $K_{n_j}$
\item $g_{j+1}=\varphi_{j+1}$ on $U\setminus K_{n_{j+1}}$
\item $g_{j}\leq\max\{h_1, ..., h_{j}\}-\delta_j/2 <f+\varepsilon \textrm{ on } U$
\item $g_j\geq h_j-\varepsilon_j\geq f \textrm{ on } K_{n_j}$.
\end{enumerate}
\end{claim}
\begin{proof}
By using Proposition \ref{properties of M(f,g)} we see that $g_{j+1}$ is a strongly convex $C^\infty$ function satisfying
$$
\max\{g_j, \varphi_{j+1}\}\leq g_j\leq \max\{g_j, \varphi_{j+1}\} +\delta_{j+1}/2.
$$
Since 
\begin{eqnarray*}
& &g_j\geq\varphi_j>h_j-\frac{\varepsilon_j}{2^{3j-3}\cdot 3}>h_{j+1}+\frac{2}{3}\varepsilon_j-\frac{\varepsilon_j}{2^{3j-3}\cdot 3} \\
& &> \varphi_{j+1}+\frac{\varepsilon_{j+1}}{2^{3(j+1)-2}\cdot 3}+\frac{2}{3}\varepsilon_j-\frac{\varepsilon_j}{2^{3j-3}\cdot 3}>\varphi_{j+1}+\delta_{j+1},
\end{eqnarray*}
Proposition \ref{properties of M(f,g)} also implies that
$$
g_{j+1}=M_{\delta_{j+1}}(g_j, \varphi_{j+1})=g_j \textrm{ on } K_{n_j},
$$
which shows $(1)$. As for $(4)$, we have
$$
g_{j}\geq\varphi_{j}>h_{j}-\frac{\varepsilon_{j}}{2^{3j-3}\cdot 3}>h_{j}-\varepsilon_{j}\geq f \textrm{ on } K_{n_{j}}.
$$
We show the rest of these properties by induction on $j$. On $U\setminus K_{n_2}$, we have
$$
\varphi_2>h_2-\varepsilon_2/24>h_1-\varepsilon_1/12-\varepsilon_2/24> g_1+\varepsilon_1/6 -\varepsilon_1/12-\varepsilon_2/24 >g_1+\delta_2,
$$
so we obtain that $g_2=\varphi_2$ outside $K_{n_2}$. Assuming that $(2)$ is true for $1\leq j\leq \ell-1$, let us see that $g_{\ell+1}=\varphi_{\ell+1}$ on $U\setminus K_{n_{\ell+1}}$. On $U\setminus K_{n_{\ell+1}}$ we have
\begin{eqnarray*}
& & \varphi_{\ell+1}>h_{\ell+1}-\frac{\varepsilon_{\ell+1}}{2^{3(\ell+1)-3}\cdot 3}>h_{\ell}-\frac{\varepsilon_j}{2^{3\ell-1}\cdot 3} -\frac{\varepsilon_{\ell+1}}{2^{3(\ell+1)-3}\cdot 3} \\
& & >\varphi_{\ell}+\frac{\varepsilon_\ell}{2^{3\ell-2}\cdot 3}-\frac{\varepsilon_\ell}{2^{3\ell-1}\cdot 3} -\frac{\varepsilon_{\ell+1}}{2^{3(\ell+1)-3}\cdot 3} \\
& & =g_{\ell}+\frac{\varepsilon_\ell}{2^{3\ell-1}\cdot 3}-\frac{\varepsilon_{\ell+1}}{2^{3(\ell+1)-3}\cdot 3}\geq 
g_{\ell}+\frac{\varepsilon_{\ell+1}}{2^{3\ell-1}\cdot 3}-\frac{\varepsilon_{\ell+1}}{2^{3(\ell+1)-3}\cdot 3}>g_{\ell}+\delta_{\ell+1},
\end{eqnarray*}
hence $g_{\ell+1}=M_{\delta_{\ell+1}}(g_{\ell}, \varphi_{\ell+1})=\varphi_{\ell+1}$ outside $K_{n_{\ell+1}}$. This proves $(2)$.

Finally, for $(3)$, we have
$$
g_2 \leq \max\{g_1, \varphi_2\}+\delta_2/2 \leq\max\{h_1-\varepsilon_1/6, h_2-\varepsilon_2/48\}+\delta_2/2\leq\max\{h_1, h_2\}-\delta_2/2 <f+\varepsilon 
$$
on $U$. Assume now that property $(3)$ is true for $j$, and let us see that it is also true for $j+1$. Indeed we have
\begin{eqnarray*}
& & g_{j+1} \leq \max\{g_j, \varphi_{j+1}\}+\frac{\delta_{j+1}}{2}\leq \max\{
\max\{h_1, ..., h_{j}\}-\frac{\delta_j}{2}, \, h_{j+1}-\delta_{j+1}\}+\frac{\delta_{j+1}}{2} \\
& & \leq  \max\{h_1, ..., h_{j+1}\}-\delta_{j+1}<f+\varepsilon
\end{eqnarray*}
on $U$. This shows $(3)$.
\end{proof}

\begin{claim}\label{estimate on Dg}
For every $j\in\N$ we have $\|Dg_{j}-Df\|\leq\varepsilon$ on $K_{n_{j+1}}$.
\end{claim}
\begin{proof}
On $K_{n_{j+1}}\setminus K_{n_j}$ we know by the preceding claim that $g_{j}=\varphi_j$, so we also have
$$
\|Dg_{j}-Df\|=\|D\varphi_{j}-Df\|\leq \|D\varphi_j-Dh_j\|+\|Dh_j-Df\|\leq\frac{1}{4}\min_{y\in K_{n_{j+1}}}\varepsilon(y) +\frac{1}{4}<\varepsilon.
$$
On $K_{n_{j}}\setminus K_{n_{j-1}}$ we have, using Lemma \ref{estimate of first derivatives of smooth maxima} and the above claim,
\begin{eqnarray*}
& & \|Dg_j-Df\|\leq\frac{1}{2}\|Dg_{j-1}-D\varphi_j\|+\frac{1}{2}\|Dg_{j}+D\varphi_{j}-2Df\| \\
& & \leq \frac{1}{2}\|Dg_{j-1}-D\varphi_j\|+ \frac{1}{2}\|Dg_{j-1}-Df\| + \frac{1}{2}\|Df-D\varphi_j\| \\
& & \leq \frac{1}{2}\|D\varphi_{j-1}-Dh_{j-1}\|+\frac{1}{2}\|Dh_{j-1}-Df\|+
\frac{1}{2}\|Df-D\varphi_j\|+\frac{1}{2}\|Dg_{j-1}-Df\|+ \frac{1}{2}\|Df-D\varphi_j\| 
\\
& & \leq \|D\varphi_{j-1}-Dh_{j-1}\|+\|Df-Dh_{j-1}\|+ \|D\varphi_{j}-Dh_{j}\|+
\|Df-Dh_{j}\| \\
& & \leq \frac{1}{4}\min_{y\in K_{n_j}}\varepsilon(y) +\frac{1}{4}\varepsilon +\frac{1}{4}\min_{y\in K_{n_{j+1}}}\varepsilon(y)+\frac{1}{4}\varepsilon.
\end{eqnarray*}
On $K_{n_1}=K_1$, we have $g_2=g_1=\varphi_1$, hence
$$
\|Dg_2-Df\|\leq \|Dg_1-Dh_1\|+\|Dh_1-Df\|\leq\frac{1}{4}\min_{y\in K_{n_2}}\varepsilon(y) +\frac{1}{4}\varepsilon\leq\varepsilon.
$$
By combining this with the above properties and an obvious induction argument (using property $(1)$ of the preceding claim), we deduce that $\|Dg_{j}-Df\|\leq\varepsilon$ on $K_{n_{j+1}}$ for every $j$.
\end{proof}

Let us finally define
$$
g(x)=\lim_{j\to\infty}g_j(x), \,\,\, x\in U.
$$
Since the $g_j\in C^{\infty}(U)$ are strongly convex and satisfy $g_{j+1}=g_{j}$ on $K_{n_j}$,  $K_{n_j}\subset \textrm{int}(K_{n_{j+1}})$  for every $j$, and $U=\bigcup_{j\in\N} K_{n_j}$, it is clear that $g$ is well defined, strongly convex, and of class $C^{\infty}(U)$. From Claim \ref{estimates on g} we see that
$$
f<g<f+\varepsilon \textrm{ on } U,
$$
and from Claim \ref{estimate on Dg} we deduce that
$$
\|Dg-Df\|\leq\varepsilon \textrm{ on } U.
$$
In order to obtain a real-analytic function $\psi$ with these properties, we consider
$$
\eta(x) :=\frac{1}{4}\min\{ g(x)-f(x), \|Dg(x)-Df(x)\|, \, f(x)+\varepsilon(x)-g(x), \, \min_{|v|=1} D^{2}g(x)(v)^2 \}, \,\,\, x\in U,
    $$
which defines a strictly positive continuous function on $U$, and we apply Whitney's approximation theorem to find a real-analytic function $\psi:U\to\R$ such
that
    $$
\max\{|\psi-g|, \|D\psi-Dg|, \|D^2\psi-D^2g|\}\leq \eta.
    $$
This implies that $f<\psi<f+\varepsilon$, $\|Df-Dg\|<\varepsilon$  and
$D^2\psi\geq\frac{1}{2}D^2g>0$, so $g$ is strongly convex too.  \qed

\medskip

Let us make one final remark. One can wonder if there are analogues of Theorem \ref{C1 fine result} for $C^k$ fine approximation with $k\geq 2$. Our methods cannot be employed to answer this question, due to the following fact: if $f, g$ are $C^2$ convex functions, then in general the second derivative of $M_{\delta}(f,g)$ blows up as $\delta$ goes to $0$.

\end{document}